\documentclass{scrartcl}

\usepackage{latexsym}
\usepackage{amsmath}
\usepackage{amsthm}
\usepackage{amssymb}
\usepackage{stmaryrd}

\newtheorem{lemma}{Lemma}[section]
\newtheorem{theorem}[lemma]{Theorem}

\newtheorem{proposition}[lemma]{Proposition}

\newtheorem{definition}[lemma]{Definition}

\renewcommand{\phi}{\varphi}
\renewcommand{\epsilon}{\varepsilon}
\newcommand{\UCP}{\mathrm{UCP}}
\newcommand{\n}{\mathrm{neg}}
\newcommand{\entropy}{\mathrm{H}}
\newcommand{\bin}{\{0,1\}}

\newcommand{\Turan}{Tur{\'a}n}
\newcommand{\IDR}{\mathrm{IDR}}
\newcommand{\litVar}[1]{\llbracket #1\rrbracket}

\title{On CNF formulas irredundant with respect to unit clause propagation}
\author{Petr Savick{\'y}\\
\normalsize Institute of Computer Science of the Czech Academy of Sciences\\
\normalsize Czech Republic, e-mail: savicky@cs.cas.cz}
\date{}

\begin{document}
\maketitle

\begin{abstract}
\begin{center}\textbf{Abstract}\end{center}
\bigskip
\small
Two CNF formulas are called \emph{ucp-equivalent}, if they behave in the same
way with respect to the unit clause propagation (UCP). A formula is called
\emph{ucp-irredundant}, if removing any clause leads to a formula which
is not ucp-equivalent to the original one. As a consequence of known
results, the ratio of the size of a ucp-irredundant formula and the
size of a smallest ucp-equivalent formula is at most $n^2$, where $n$ is
the number of the variables. We demonstrate an example of a ucp-irredundant
formula for a symmetric definite Horn function which is larger than
a smallest ucp-equivalent formula by a factor $\Omega(n/\ln n)$.
Consequently,
a general upper bound on the above ratio cannot be smaller than this.
\end{abstract}

\section{Introduction}

In some contexts,
the strength of unit clause propagation (UCP) on the formula
is important and formulas representing the same function have
to be distinguished according to the behavior of UCP on them.
An important example is the set of learned clauses in a SAT solver.
Most of these clauses are implicates of the original formula so
they do not change the represented function and they are added
in order to make UCP stronger.
In order to investigate CNF formulas in such a context,
we need to compare the strength of UCP on them.
In this paper, we consider a corresponding equivalence and
irredundancy.

Two CNF formulas are called \emph{ucp-equivalent},
if for every partial assignment $\alpha$, UCP derives the same
consistent set of literals from the two formulas together with
$\alpha$ or we get a contradiction in both cases. Considering
total assignments $\alpha$, one can verify that ucp-equivalent
formulas represent the same function, so they are equivalent.

Let us present a few examples demonstrating the notion of
ucp-equivalence. Formulas
\begin{eqnarray*}
\phi_1 & = & (a \vee c) \wedge (b \vee c) \\
\phi_2 & = & (a \vee c) \wedge (\neg a \vee b \vee c)
\end{eqnarray*}
are equivalent, but not ucp-equivalent, since UCP derives
the literal $c$ from the formula $\phi_1 \wedge \neg b$ and
not from the formula
$\phi_2 \wedge \neg b$. Formula $\phi_2$ is not prime, since
the literal $\neg a$ can be removed without changing the represented
function. Consider prime formulas
\begin{eqnarray*}
\phi_3 & = & (\neg a \vee b) \wedge (\neg a \vee c) \wedge (\neg b \vee \neg c \vee d) \\
\phi_4 & = & \phi_3 \wedge (\neg a \vee d)
\end{eqnarray*}
which are equivalent, since $\neg a \vee d$ is a consequence of $\phi_3$.
These formulas are not ucp-equivalent, since UCP on
$\phi_4 \wedge \neg d$ derives $\neg a$ while not on
$\phi_3 \wedge \neg d$.

Ucp-equivalent formulas have to be much closer to each other than
equivalent formulas. However, ucp-equivalent formulas may not contain
any common clause. An example of a pair of such formulas is
\begin{eqnarray*}
\phi_5 & = & (\neg a \vee b) \wedge (\neg b \vee c) \wedge (\neg c \vee a) \\
\phi_6 & = & (\neg a \vee c) \wedge (\neg c \vee b) \wedge (\neg b \vee a) \;.
\end{eqnarray*}
Without an additional partial assignment, UCP derives no new literal
on both these formulas.
After adding a positive literal on any of the variables $a$, $b$, or $c$,
UCP derives all the remaining positive literals on both the formulas
and the same is true for the negative literals. More generally,
consider any finite set of variables and two conjunctions of implications
between pairs of these variables which both form a strongly connected
directed graph. An implication $x \to y$ is equivalent to the clause
$\neg x \vee y$, so the conjunctions can be considered as CNF formulas.
Using the same argument as above, such formulas are always ucp-equivalent,
although the graphs can be quite different.

A formula is called \emph{ucp-irredundant},
if removing any clause changes the represented function or makes UCP
on the formula weaker. In both these cases,
the obtained formula is not ucp-equivalent to the original one.
All the formulas $\phi_1$, \ldots, $\phi_6$
are ucp-irredundant. Removing any clause from them except of $\neg a \vee d$
from $\phi_4$ changes the represented function. Removing the clause
$\neg a \vee d$ does not change the function represented by $\phi_4$,
but makes UCP on the formula weaker.

UCP on a CNF formula $\phi$ of $n$ variables can be precisely represented
using a specific Horn formula on $2n$ variables which we call
implicational dual rail encoding of $\phi$ and denote $IDR(\phi)$.
The construction is a Horn part of a formula used
in \cite{BKNW09,BJMSM12} to simulate UCP. The variables
of this formula represent literals on the variables of the original
formula. Such a representation is known as a dual rail encoding, see,
for example \cite{BBIMM18,IMM17,MIBMB19}, where it is used to reduce SAT
problem to MaxSAT for a specific Horn formula. The implicational dual rail
encoding of a formula $\phi$ is also a Horn formula, however, it is quite
different. It contains $k$ definite Horn clauses for every
clause of length $k$ in $\phi$, since each of the $k$ possible steps of UCP
in the clause is represented by a separate Horn clause or, equivalently,
by an implication.

Let us point out that representation of UCP using a Horn formula
is used in \cite{BKNW09}
to characterize the size of a CNF decomposition of a propagator
known also as a CNF encoding satisfying generalized arc consistency
for an arbitrary constraint using the
size of a specific monotone Boolean circuit up to a polynomial relation.
In \cite{BJMSM12} a formula whose Horn part represents UCP in some
other formula is used to construct a CNF
encoding of a disjunction of several CNF formulas preserving the
strength of UCP of the input formulas.
In \cite{Kucera22}, implicational dual rail encoding is used to construct
an algorithm computing a propagation complete formula using
results on learning definite Horn formulas.

Two formulas are ucp-equivalent if and only if their implicational
dual rail encodings
are equivalent in the sense to represent the same Horn
function. It is a well-known result \cite{HK93} that the ratio
of the size of a Horn formula irredundant with respect
to representing a given function and the size of a smallest
equivalent formula is at most the number of the variables.

Situations when the size of an irredundant and a smallest equivalent
structure are close to each other are rare in complexity.
In particular, this is not true for
CNF formulas with respect to equivalence. An unsatisfiable formula
of $n$ variables consisting of all $2^n$ clauses of length $n$ is
minimally unsatisfiable, since removing any of its clauses
makes it satisfiable. Hence, this formula is
irredundant in the sense that removing any clause changes the
represented function. However, the smallest equivalent formula
consists only of an empty clause.

Horn formulas represent a very narrow class of Boolean functions.
Using the implicational dual rail encoding we can transform the result for
Horn formulas to ucp-irredundant formulas representing an
arbitrary Boolean function. In this case, we consider
the ratio of the size of any ucp-irredundant formula and
the size of a smallest ucp-equivalent formula.
Since implicational dual rail encoding changes the number of the clauses, we get
a bound which is worse than in the case of Horn formulas.
Namely, the upper bound is $2n^2$, where $n$ is the number of the
variables.
A slightly better bound $n^2$ can be obtained by a proof used
in \cite{KS20} to prove the same bound for propagation complete
formulas since the assumption of propagation completeness is
used only to guarantee ucp-equivalence.
Moreover, it appears that the factor $n^2$ holds both for the
number of clauses and the length of the formula.

An interesting question is what is the maximum value of the
ratio of the size of any ucp-irredundant formula and the size
of a smallest ucp-equivalent formula in terms of the number
of the variables. It is plausible to assume that the upper
bound $n^2$ can be improved, however, this is an open problem.
The main result of the paper is an example of a ucp-irredundant
formula $\phi^\ell$ of $n$ variables which is larger than the
smallest ucp-equivalent formula $\phi^*$ by a factor $\Omega(n/\ln n)$.
It follows that a general upper bound on the ratio cannot be
smaller than this.

The formulas used in the example represent a symmetric definite
Horn function
and the construction of the formula $\phi^\ell$ is easy.
The hard part of the proof is to derive an upper bound on the
size of $\phi^*$ and we show that the size of $\phi^*$ is closely
related to the size of a specific covering system.
Namely, for $n$ and $k = n/2 + O(1)$, we use a system of
$(k+1)$-subsets of an $n$-set such that every $k$-set is
contained in at least one member of the system. The proof
of the needed properties of such systems and their relationship
to the formulas uses probabilistic method and takes a significant
part of the paper.
The proof is finished using a known upper bound
on the size of such a system from \cite{ES74} or a different
upper bound from \cite{Sid97,FR85}.
It is an interesting question whether the lower bound
$\Omega(n/\ln n)$ or an even better bound can be proven
using a different type of functions and without the use
of probabilistic method.

Section~\ref{sec:ucp-equiv-irred} defines ucp-equivalence and
ucp-irredundancy and Section~\ref{sec:results-formulation}
formulates the results.
The property of the formulas needed for the results is reformulated
in Section~\ref{sec:combinatorial-characterization} as a property
of a hypergraph. The construction of the required hypergraph
is described in Section~\ref{sec:size-of-hypergraph} using
an auxiliary result from Section~\ref{sec:random-permutations}.
The proofs of the results are summarized in
Section~\ref{sec:results-proof}. Conclusion and questions
for further research are formulated in Section~\ref{sec:conclusion}.

\section{Preliminaries}

\subsection{CNF formulas}
\label{sec:CNF-formulas}

A literal is a Boolean variable or its negation,
a clause is a disjunction of literals on different variables,
and a CNF formula is a conjunction of clauses. An empty
clause denoted $\bot$ is a clause containing no literals and
represents a contradiction.
A unit clause is a clause containing exactly one literal.
Since we consider only CNF formulas, we call them formulas
for simplicity.

A formula is considered as a set of clauses and a clause is
considered as a set of literals. In particular,
we use set operations on formulas and clauses.
If $A$ is a set of variables, let $\n(A) = \{\neg x \mid x \in A\}$
be a set of literals used typically as a part of a clause.
The size
$|\phi|$ of a formula $\phi$ is the number of its clauses.
The length of a clause
is the number of the literals in it. The length $\|\phi\|$
of a formula $\phi$ is the sum of the lengths of its clauses.

A clause $C$ is an implicate of a formula $\phi$, if every
assignment of the variables satisfying $\phi$ satisfies $C$.
A clause $C$ is a prime implicate of a formula $\phi$,
if it is an implicate of $\phi$ and no proper subset of $C$ is
an implicate of $\phi$.

A formula is a Horn formula, if every clause contains at most
one positive literal, and a formula is a definite Horn formula,
if every clause contains exactly one positive literal.

A partial assignment of the variables is a partial map from the
variables to the set $\{0,1\}$ understood as a conjunction of identities
of the form $x=a$, where $x$ is a variable and $a \in \bin$.
A partial assignment will be represented by a conjunction of the literals
equivalent to these identities. For example, the conjunction of identities
$x_1=1 \wedge x_2=0$ is equivalent to the conjunction of literals
$x_1 \wedge \neg x_2$.

Unit clause propagation (UCP) is a process of deriving literals or
a contradiction implied by the formula using the following process.
Assume a formula $\phi$ and let us initialize a set
$U$ as the set of clauses of length at most one contained in
$\phi$. Clauses of length one will be identified with the literals
contained in them.
All clauses in the initial set $U$ are
elements of $\phi$, so they are clearly consequences of $\phi$.
Assume a clause $C \in \phi$, and a literal $l \in C$.
If $U$ contains the negations of all the literals in
$C \setminus \{l\}$, then every satisfying assignment of $\phi$
satisfies $l$ and, hence, $l$ is a consequence of $\phi$.
Similarly, if $U$ contains the negation of all the literals
in $C$, then $\bot$ is a consequence of $\phi$. Unit clause propagation
on a formula $\phi$ is the process of iteratively extending $U$ by
literals and possibly $\bot$ derived as consequences of
this simple type. The process terminates if no further literal
or $\bot$ can be derived or if $\bot \in U$.

\begin{definition}
For every CNF formula $\phi$, let $\mathrm{UCP}(\phi)$ be
\begin{itemize}
\item
the set of literals contained in $\phi$ as unit clauses or derived by
unit clause propagation on $\phi$, if the contradiction is not reached,
\item
$\{\bot\}$ otherwise.
\end{itemize}
\end{definition}
One can show that $\mathrm{UCP}(\phi)$ does not depend on the
order of the derivations performed by the process of UCP.
If $\mathrm{UCP}(\phi) = \{\bot\}$, then $\phi$ is unsatisfiable.
The opposite is not true. For example, the formula
$
\phi_0 =
(x_1 \vee x_2) \wedge
(x_1 \vee \neg x_2) \wedge
(\neg x_1 \vee x_2) \wedge
(\neg x_1 \vee \neg x_2)
$
is unsatisfiable, however, $\UCP(\phi_0) = \emptyset$.

A non-empty partial assignment $\alpha$ can be used to \emph{simplify}
a formula $\phi$ to a formula $\phi'$ as follows. For every
$l \in \alpha$, every clause containing $l$ is removed and the negated
literal $\neg l$ is removed from all other clauses. One can
verify that $\UCP(\phi \wedge \alpha) = \UCP(\phi' \wedge \alpha)$.

\subsection{Covering and {\Turan} numbers}
\label{sec:covering-designs}

A $k$-set or a $k$-subset is a set or a subset of size $k$.
We use known bounds on covering numbers $C(n,r,k)$ \cite{ES74,GKP95}
and the closely related {\Turan} numbers $T(n,\ell,t)$ \cite{FR85,Sid97}.
Since there are different conventions to define these numbers, we
recall the definitions which we use.

\begin{definition}
For integers $n > r > k > 0$, let the covering number $C(n,r,k)$ be the
smallest size of
a system of $r$-subsets of an $n$-set such that every $k$-subset
is contained in at least one set in the system.
\end{definition}

\begin{definition}
For integers $n > \ell > t > 0$, let the {\Turan} number $T(n,\ell,t)$ be the
smallest size of
a system of $t$-subsets of an $n$-set such that every $\ell$-subset
contains at least one set in the system.
\end{definition}

Taking complements of the sets, we obtain $C(n,r,k) = T(n,n-k,n-r)$.
Note that \Turan's theorem implies a precise value of $T(n,\ell,2)$,
if we minimize the number of edges to be removed from a complete graph
on $n$ vertices in order to eliminate all complete subgraphs on $\ell$
vertices.

Let us recall several known results on covering numbers in the
case $r=k+1$ which we use. Clearly,
\begin{equation} \label{eq:trivi-lower-bound}
C(n,k+1,k) \ge \frac{1}{k+1} {n \choose k} \;.
\end{equation}
If $k$ is fixed and $n \to \infty$, then by \cite{R85}, we have
$$
C(n,k+1,k) = (1 + o(1)) \frac{1}{k+1} {n \choose k} \;.
$$
We are mainly interested in cases when $k=n/2+O(1)$ for which
only less precise bounds are known. Let $\mu(n,k)$ be such that
\begin{equation} \label{eq:define-mu-C}
C(n, k+1, k) = \frac{1}{k+1} {n \choose k} \mu(n, k)
\end{equation}
and, equivalently,
\begin{equation} \label{eq:define-mu-T}
T(n, t+1, t) = \frac{1}{t+1} {n \choose t} \mu(n, n-t-1) \;.
\end{equation}
For an upper bound on $\mu(n,k)$, let us recall the following.

\begin{theorem}[a special case of Theorem 13.4 from \cite{ES74}] \label{thm:ES74}
For every $n > k \ge 2$, we have
\begin{equation} \label{eq:covering-number-ub}
C(n,k+1,k) \le \frac{1}{k+1} {n \choose k} (1 + \ln(k+1))
\end{equation}
\end{theorem}

Together with (\ref{eq:trivi-lower-bound}), this implies
\begin{equation} \label{eq:mu-upper-1}
1 \le \mu(n,k) \le 1 + \ln(k+1)
\end{equation}
for every $n > k \ge 2$.
An incomparable upper bound on $\mu(n,k)$ follows from the bound
\cite{Sid97} on {\Turan} numbers which improves \cite{FR85} by a factor
of $2$. By this bound, we have
\begin{equation} \label{eq:bound-on-turan}
T(n, t+1, t) \le \frac{(1+o(1)) \ln t}{2t} {n \choose t}
\end{equation}
for $n \to \infty$ and $t \to \infty$.
Combining (\ref{eq:bound-on-turan}) with (\ref{eq:define-mu-T}),
substituting $k = n-t-1$ and assuming $n-k \to \infty$, we obtain
\begin{equation} \label{eq:mu-upper-2}
\mu(n,k) \le \frac{1+o(1)}{2} \ln (n-k-1) \;.
\end{equation}
Bounds (\ref{eq:mu-upper-1}) and (\ref{eq:mu-upper-2}) are complementary
in the sense that (\ref{eq:mu-upper-1}) is better, if $k$ is small,
while (\ref{eq:mu-upper-2}) is better, if $k$ is close to $n$.
Although the two bounds are incomparable in general, they are close
to each other for $n$ and $k = n/2 + O(1)$ used in this paper.

In Section~\ref{sec:results-proof} it is advantageous to rely on
inequalities for the function $\mu(n,k)$ instead of inequalities
for the covering numbers. The quantity $\mu(n,k)$ is logarithmically
bounded and using it allows us to consider this relatively small factor
independently of the main terms of the asymptotic estimate. Namely,
we will need the following.

\begin{lemma}
For every $n > k > 1$, we have
\begin{eqnarray}
\mu(n-1, k-1) & \le & \mu(n, k) \label{eq:mu-1}\\
\frac{n-k+1}{n+1}\mu(n, k) & \le & \mu(n+1, k)\label{eq:mu-2}
\end{eqnarray}
\end{lemma}

\begin{proof}
Both inequalities are direct consequencies of inequalities for
the covering numbers. In particular, by \cite{Sch64,GKP95}, we have
$$
\frac{n}{k+1} C(n-1, k, k-1) \le C(n, k+1, k)
$$
which implies (\ref{eq:mu-1}) using (\ref{eq:define-mu-C}).

Assume a system of $(k+1)$-subsets
of a set $X$ of $n+1$ points of size $C(n+1,k+1,k)$ and covering
all $k$-subsets. For any $x \in X$, one can create a system of
$(k+1)$-subsets of $X \setminus \{x\}$ covering all $k$-subsets
by replacing each occurence of $x$ in a $(k+1)$-tuple in the
system by any element of $X \setminus \{x\}$ not contained in
the $(k+1)$-tuple. Hence, we have
$$
C(n, k+1, k) \le C(n+1, k+1, k)
$$
implying (\ref{eq:mu-2}) using (\ref{eq:define-mu-C}).
\end{proof}

\section{Ucp-equivalence and ucp-irredundancy}
\label{sec:ucp-equiv-irred}

Let us introduce the following equivalence relation on CNF formulas.

\begin{definition}
Formulas $\phi_1$ and $\phi_2$ are \emph{ucp-equivalent} if for
every partial assignment $\alpha$, we have
$\UCP(\phi_1 \wedge \alpha) = \UCP(\phi_2 \wedge \alpha)$.
\end{definition}

By considering total assignments represented by $\alpha$ containing
exactly one literal (positive or negative) on each variable, one can
prove the following.

\begin{proposition} \label{prop:equivalence}
If two formulas are ucp-equivalent, then they are equivalent
in the sense to represent the same function.
\end{proposition}

The notion of ucp-equivalence appeared implicitly in analysis of
propagation complete (PC) formulas \cite{BMS12,BBCGKV13},
so let us recall their definition.
A formula $\phi$ is called \emph{propagation complete}, if for every
partial assignment $\alpha$ of the variables, UCP on
$\phi \wedge \alpha$ derives a contradiction, if $\phi \wedge \alpha$
is inconsistent, and all literals implied by $\phi \wedge \alpha$
otherwise. Since UCP is logically sound, a formula is propagation complete,
if the strength of UCP on it is maximal among all formulas representing
the same function. In particular, every two
equivalent PC formulas are ucp-equivalent, since UCP on them
is controlled by the function they represent. One can also
verify that a formula is PC if and only if it is ucp-equivalent
to the set of all its (prime) implicates.

Let us recall the notion of an \emph{absorbed clause} introduced
in \cite{BMS12}. Negation of a clause is the conjunction of the
negations of its literals.
A clause $C$ is absorbed by a formula $\phi$, if for every $l \in C$,
UCP on the formula $\phi \wedge \neg (C \setminus \{l\})$
derives $l$ or a contradiction.
A formula is called minimal PC formula in \cite{BMS12} if none of its
clauses is absorbed by the remaining ones.
This does not guarantee that no smaller equivalent PC formula
exists, however, the formula is irredundant with respect to the
behavior of UCP. We use a straightforward generalization of this
type of irredundancy obtained by removing the requirement that the
formula is propagation complete.

\begin{definition}
A formula $\phi$ is \emph{ucp-irredundant} if any of the following
equivalent conditions is satisfied:
\begin{itemize}
\item
None of the clauses $C \in \phi$ is absorbed by $\phi \setminus \{C\}$.
\item
For every $C \in \phi$, the formulas $\phi$ and $\phi \setminus \{C\}$
are not ucp-equivalent.
\end{itemize}
\end{definition}

The notion of absorption can naturally be extended to formulas.

\begin{definition}
A formula $\phi_1$ is absorbed by the formula $\phi_2$, which will
be denoted $\phi_1 \le_{ucp} \phi_2$, if every clause of $\phi_1$
is absorbed by $\phi_2$.
\end{definition}

One can verify that $\phi_1 \le_{ucp} \phi_2$ if and only if for
every partial assignment $\alpha$, we have
$\UCP(\phi_1 \wedge \alpha) \subseteq \UCP(\phi_2 \wedge \alpha)$
or $\UCP(\phi_2 \wedge \alpha) = \{\bot\}$. It follows that
the relation $\le_{ucp}$ is a preorder and any
two formulas $\phi_1$ and $\phi_2$ are ucp-equivalent if and only
if $\phi_1 \le_{ucp} \phi_2$ and $\phi_2 \le_{ucp} \phi_1$.
In particular,
$\le_{ucp}$ defines a partial order on the classes of ucp-equivalent
formulas and the class of PC formulas is the maximal class in this order.
Since, by definition, $\phi_1 \le_{ucp} \phi_2$ can be tested in polynomial
time, also ucp-equivalence can be tested in polynomial time.

By the result \cite{KS20}, if $\phi$ is a ucp-irredundant PC formula,
then the ratio of its size and the size of a smallest PC formula
equivalent to $\phi$ is
at most $n^2$, where $n$ is the number of the variables. The assumption
that the two formulas are PC is used only to guarantee that they behave
in the same way with respect to UCP or, equivalently, that the two
formulas are ucp-equivalent. It follows that the upper bound on the
ratio of their sizes can be generalized to non-PC formulas
as follows. The statement
is the strongest if $\phi^*$ is the smallest formula ucp-equivalent to
$\phi$, however, the proof does not use this assumption.

\begin{proposition} \label{prop:upper-bound}
If $\phi$ is a ucp-irredundant formula on $n$ variables and $\phi^*$
is ucp-equivalent to $\phi$, then
\begin{equation} \label{eq:1a}
|\phi| \le n^2 |\phi^*|
\end{equation}
and
\begin{equation} \label{eq:1b}
\|\phi\| \le n^2 \|\phi^*\|
\end{equation}
\end{proposition}

\begin{proof}
Every clause of $\phi^*$ is absorbed by $\phi$. Let $\phi'$
be the set of clauses in $\phi$ needed for this for all clauses of
$\phi^*$. For each pair $l \in C \in \phi^*$, we need at most
$n$ clauses of $\phi$ to derive $l$ or a contradiction from
$\phi \wedge \neg (C \setminus \{l\})$. This implies $|\phi'| \le n \|\phi^*\|$.
Since $\phi' \subseteq \phi$, we have
$$
\phi \le_{ucp} \phi^* \le_{ucp} \phi' \le_{ucp} \phi
$$
which implies that the formula $\phi'$ is ucp-equivalent to $\phi$.
Since $\phi$ is ucp-irredundant, $\phi'$ is equal to $\phi$.
Hence, we have $|\phi| \le n \|\phi^*\|$ implying
$$
|\phi| \le n \|\phi^*\| \le n^2 |\phi^*|
$$
and
$$
\|\phi\| \le n |\phi| \le n^2 \|\phi^*\|
$$
as required.
\end{proof}

Using part of a construction used in \cite{BKNW09,BJMSM12},
UCP on a CNF formula can be represented by the following Horn formula.
Assume, $\phi = C_1 \wedge \ldots \wedge C_m$ is a formula on variables
$X=\{x_1, \ldots, x_n\}$. For every literal $l$ on a variable from $X$
introduce a variable $\litVar{l}$. A total assignment of these variables
satisfying $\neg \litVar{x_i} \vee \neg \litVar{\neg x_i}$ for every
variable $x_i \in X$ encodes the set of the satisfied literals under
a partial assignment of the variables $X$.
The \emph{implicational dual rail encoding} of $\phi$ is
$$
\IDR(\phi) = \IDR(C_1) \wedge \ldots \wedge \IDR(C_m) \wedge
\bigwedge_{i=1}^n (\neg \litVar{x_i} \vee \neg \litVar{\neg x_i})
$$
where for a clause $C = l_1 \vee \ldots \vee l_k$,
$$
\IDR(C) = \bigwedge_{i=1}^k \left(\bigwedge_{j\not=i} \litVar{\neg l_j} \to \litVar{l_i}\right)
$$
The set of models of $\IDR(\phi)$ is precisely the set of encodings
of partial assignments of the variables $X$ which are closed under UCP.
This implies the following.

\begin{proposition}
Two formulas $\phi_1$ and $\phi_2$ are ucp-equivalent if and only if
the Horn formulas $\IDR(\phi_1)$ and $\IDR(\phi_2)$ are equivalent.
\end{proposition}

\section{The results}
\label{sec:results-formulation}

The main result is the following separation between the size
of a ucp-irredundant formula and a smallest ucp-equivalent formula.

\begin{theorem} \label{thm:separation}
For every sufficiently large $n$ and an integer $k=n/2 + O(1)$, there
is a ucp-irredundant formula $\phi^\ell$ of $n$ variables and of size
$$
|\phi^\ell| = (k+1) {n-1 \choose k} = \Theta(n) {n \choose k}
$$
and such that a smallest ucp-equivalent formula $\phi^*$
has size
$$
|\phi^*| \le 6(k+1) C(n, k+1, k) = O(\ln n) {n \choose k}
$$
In particular, $|\phi^\ell| / |\phi^*| = \Omega(n / \ln n)$.
\end{theorem}

For a proof, we investigate the size of ucp-irredundant formulas
ucp-equivalent to the following definite Horn formula $\Psi_{n,k}$.

\begin{definition}
For any integers $n$, $k$, such that $1 \le k \le n-1$,
let $\Psi_{n,k}$ be the CNF formula consisting of all clauses on
variables $X = \{x_1, \ldots, x_n\}$ containing $k$ negative literals
and one positive literal.
\end{definition}

Clearly, $|\Psi_{n,k}| = (n-k) {n \choose k}$.
A Boolean function is called symmetric, if it is invariant under
any permutation of its variables. Note that formula $\Psi_{n,k}$
represents the symmetric Boolean function $f_{n,k}$ of $n$ variables
defined by
\begin{equation} \label{eq:function}
f_{n,k}(x_1, \ldots, x_n) = 1 \Leftrightarrow
\left( \sum_{i=1}^n x_i < k \right) \vee
\left( \sum_{i=1}^n x_i = n \right)
\end{equation}
and, moreover, $\Psi_{n,k}$ is the set of all prime implicates
of this function.

The proof of Theorem~\ref{thm:separation} relies on the upper
bound on the size of a smallest formula ucp-equivalent to $\Psi_{n,k}$
implied by the following proposition. The lower bound on this size
is included in order to demonstrate that the upper bound cannot be
improved more than by a logarithmic factor.

\begin{proposition} \label{prop:size-of-theta}
For every $1 \le k \le n-1$, the minimum size (number of clauses)
of a CNF formula ucp-equivalent to $\Psi_{n,k}$ is at least
$$
n\, C(n-1, k, k-1) = {n \choose k} \mu(n-1, k-1)
\ge {n \choose k}
$$
and if $k=n/2+O(1)$ then it is at most
$$
6(k+1) C(n, k+1, k) = 6\, {n \choose k} \mu(n, k)
= O(\ln n) {n \choose k}\;.
$$
\end{proposition}

The proof of Theorem~\ref{thm:separation} and
of Proposition~\ref{prop:size-of-theta}
are summarized in Section~\ref{sec:results-proof}
using constructions presented in this and the following sections.
The lower bound in Proposition~\ref{prop:size-of-theta} is proven
in Section~\ref{sec:combinatorial-characterization}. For the
proof of the remaining statements, we construct for infinitely many
pairs $(n, k)$ ucp-irredundant formulas $\phi^\ell$
(Section~\ref{sec:combinatorial-characterization}) and
$\phi^*$ (Section~\ref{sec:results-proof}) ucp-equivalent
to $\Psi_{n,k}$, such that $\phi^\ell$ is as large as possible
and $\phi^*$ is as small as possible. The formula $\phi^*$ which
meets the upper bound in Proposition~\ref{prop:size-of-theta}
and the formula $\phi^\ell$ are then used to prove
Theorem~\ref{thm:separation}.

Each of the formulas $\phi^\ell$ and $\phi^*$ will be constructed
as a hypergraph formula which simplifies their construction
considerably. The lower bound in Proposition~\ref{prop:size-of-theta}
is valid for an arbitrary subset of $\Psi_{n,k}$. This implies
that using a hypergraph formula $\phi^*$ cannot worsen the upper
bound on $|\phi^*|$ by more than a logarithmic factor compared
to the minimum size for unrestricted formulas.

\begin{definition}[\cite{BBM23a,BBM23b}] \label{def:use-undirected-hg}
For every $(k+1)$-uniform hypergraph $H$ whose vertices are the
variables $X$, the \emph{hypergraph formula} $\theta(H)$ consists of
the clauses $\n(e \setminus \{x\}) \cup \{x\}$ for all hyperedges
$e \in H$ and all variables $x \in e$.
\end{definition}

Clearly, for every $H$ as above,
$\theta(H) \subseteq \Psi_{n,k}$ and $|\theta(H)| = (k+1)|H|$.

\section{A combinatorial characterization}
\label{sec:combinatorial-characterization}

In order to construct the formula $\phi^*$, we are looking for
a formula ucp-equivalent to $\Psi_{n,k}$ and as small as possible.
According to Lemma~\ref{lem:prime} below, we can
consider only prime formulas (containing only prime implicates)
without loss of generality.
Since the formula $\Psi_{n,k}$ is equal to the set of all its
prime implicates, such formulas are subsets of $\Psi_{n,k}$.
We present two combinatorial characterizations of the subsets of
$\Psi_{n,k}$ which are ucp-equivalent to $\Psi_{n,k}$. The first of them
is valid for an arbitrary subset of $\Psi_{n,k}$ and is used for a lower
bound on the size of an arbitrary formula ucp-equivalent to $\Psi_{n,k}$.
The other characterization is valid for subsets of $\Psi_{n,k}$ in the
form of a hypergraph formula $\theta(H)$ defined above and is used
later for a corresponding upper bound.

\begin{lemma} \label{lem:prime}
If $\phi$ is an arbitrary CNF formula ucp-equivalent to $\Psi_{n,k}$,
then there is a prime formula $\phi'$ ucp-equivalent to $\Psi_{n,k}$,
such that $|\phi'| \le |\phi|$ and $\|\phi'\| \le \|\phi\|$.
\end{lemma}

\begin{proof}
Let $\phi$ be ucp-equivalent to $\Psi_{n,k}$.
By Proposition~\ref{prop:equivalence}, $\phi$ represents the same
function as $\Psi_{n,k}$, so these two formulas have the same set of
prime implicates. Moreover, $\Psi_{n,k}$ is constructed as the set
of all prime implicates of the function (\ref{eq:function}).
Hence, $\Psi_{n,k}$ is propagation complete and, by assumption,
also $\phi$ is propagation complete. This means that UCP on $\phi$
is the strongest possible among formulas representing the function
(\ref{eq:function}).

If $C \in \phi$ is not a prime implicate of $\phi$, then replacing
it by a prime subimplicate can only make UCP stronger. Since $\phi$
is propagation complete, this replacement leads to a formula
ucp-equivalent to $\phi$ and $\Psi_{n,k}$. Repeating this, we
replace all non-prime clauses of $\phi$ by prime implicates and
possibly remove redundant clauses. The resulting formula $\phi'$
satisfies the statement.
\end{proof}

Every clause of $\Psi_{n,k}$ contains $k$ negative literals.
A partial assignment setting at most $k-2$ variables to $1$
and possibly some variables to $0$ simplifies $\Psi_{n,k}$ to a formula where
every clause contains at least $2$ negative literals. This implies
that UCP on $\Psi_{n,k}$ and such a partial assignment does not
derive any new literal and the same is true for every subset
of $\Psi_{n,k}$. It follows that when considering UCP on subsets
of $\Psi_{n,k}$ it is sufficient to consider partial assignments
setting at least $k-1$ variables to $1$ and possibly some
variables to $0$.

Moreover, it appears that in order to
characterize subsets $\phi$ of $\Psi_{n,k}$ that are ucp-equivalent
to $\Psi_{n,k}$, it is reasonable to split such an assignment
into an initial setting of exactly $k-1$ variables to $1$,
determine the set of binary clauses in the corresponding simplified formula,
and then consider a further setting of one other variable,
say $x_\ell$, to $0$ or $1$. Depending on whether $x_\ell$ is
set to $0$ or $1$, UCP on $\phi$ should derive all negative or
all positive literals on the remaining variables, since this is
true for UCP on the formula $\Psi_{n,k}$.

The binary clauses obtained by simplification after setting of
exactly $k-1$ variables
to $1$ necessarily contain one positive and one negative
literal and, hence, represent an implication. We represent this
set of implicatins by the following directed graph.

\begin{definition} \label{def:G-directed}
Assume $\phi \subseteq \Psi_{n,k}$ and a $(k-1)$-set $A \subseteq X$.
Let $G^d(\phi, A)$ be the directed graph on $X \setminus A$ defined as
$$
G^d(\phi, A) = \{(x_i, x_j) \mid \n(A) \cup \{\neg x_i, x_j\} \in \phi\} \;.
$$
\end{definition}

Let $\phi$ and $A$ be as in Definition~\ref{def:G-directed} and
let $\alpha = \bigwedge_{x \in A} x$ be a partial assignment
setting all variables in $A$ to $1$. Then UCP on $\phi \wedge \alpha$
derives no additional literals, since $\alpha$ extended by literals
of the same polarity (positive or negative) on all the remaining
variables is a satisfying assignment of $\Psi_{n,k}$ and $\phi$.
If also a variable not in $A$ is set, then UCP can be characterized by
$G^d(\phi, A)$ as follows.

\begin{lemma} \label{lem:directed-graph-G-A}
Assume $\phi \subseteq \Psi_{n,k}$, a $(k-1)$-set $A \subseteq X$, and
let $\alpha = \bigwedge_{x \in A} x$.
Consider a partial assignment $\beta \in \{\neg x_\ell, x_\ell\}$
setting a variable $x_\ell \in X \setminus A$ to $0$ (the negative
literal) or to $1$ (the positive literal).
Then $\UCP(\phi \wedge \alpha \wedge \beta)$ is consistent and we have
\begin{itemize}
\item
$\UCP(\phi \wedge \alpha \wedge \neg x_\ell)$ extends $\alpha$ precisely
by setting to $0$ the variables $x_i$ such that $G^d(\phi, A)$ contains
a path from $x_i$ to $x_\ell$.
\item
$\UCP(\phi \wedge \alpha \wedge x_\ell)$ extends $\alpha$ at least by
setting to $1$ the variables $x_i$ such that $G^d(\phi, A)$ contains
a path from $x_\ell$ to $x_i$.
\end{itemize}
\end{lemma}

\begin{proof}
By (\ref{eq:function}), setting variables in $A$ to $1$ and all
variables in $X \setminus A$ to $0$ or all of them to $1$ is
consistent with $f_{n,k}$. Hence, it is consistent also with
the formula $\Psi_{n,k}$ representing $f_{n,k}$ and its subset
$\phi$. It follows that $\UCP(\phi \wedge \alpha \wedge \beta)$
is consistent.

When simplifying a subset of $\Psi_{n,k}$ using $\alpha$, the clauses
containing a positive literal on a variable from $A$ become satisfied and,
hence, are removed. Since $|A|=k-1$ and every clause of $\phi$ contains
$k$ negative literals, the clauses remaining in the
simplified formula contain one positive
and at least one negative literal. In particular, this is a definite
Horn formula without unit clauses. In such a formula, setting
a variable $x_\ell \in X \setminus A$ to $0$ can derive a value of
a new variable by UCP only using a clause containing the positive literal
$x_\ell$ and exactly one negative literal. This negative literal
is the derived one. Since the derived value of the variable in this
literal is $0$, UCP can derive only negative literals even if the
propagation is iterated.

Simplifying a clause $C \in \phi$ by $\alpha$ produces a clause with
one negative literal only if $C=\n(A) \cup \{\neg x_i, x_j\}$ for some
$i,j$ and the resulting clause is $\{\neg x_i, x_j\}$. Hence, the
clauses of the formula $\phi$ simplified by $\alpha$ which are
relevant for propagating assignments of variables
in $X \setminus A$ to $0$ form the 2-CNF formula
$$
\phi_A = \{\{\neg x_i, x_j\} \mid (x_i, x_j) \in G^d(\phi, A)\} \;.
$$
If $x_j=0$, then UCP on $\phi_A$ derives $x_i=0$ in one step
if and only if $\{\neg x_i, x_j\} \in \phi_A$ and this is
equivalent to $(x_i, x_j) \in G^d(\phi, A)$. Hence, propagation of
zeros in $\phi_A$ is precisely represented by the reverse reachability
in $G^d(\phi, A)$ which implies the statement concerning
$\UCP(\phi \wedge \alpha \wedge \neg x_\ell)$.

The statement concerning $\UCP(\phi \wedge \alpha \wedge x_\ell)$ is clear,
since $\phi_A$ is a subset of the formula $\phi$ simplified by $\alpha$
and the clauses of $\phi_A$ allow propagation of values $1$ along all
directed paths in $G^d(\phi, A)$.
\end{proof}

\begin{proposition} \label{prop:propagation}
A formula $\phi \subseteq \Psi_{n,k}$ is ucp-equivalent
to $\Psi_{n,k}$ if and only if for every $(k-1)$-set $A \subseteq X$
the graph $G^d(\phi, A)$ is a strongly connected graph on $X \setminus A$.
\end{proposition}

\begin{proof}
Assume, $\phi \subseteq \Psi_{n,k}$ is ucp-equivalent
to $\Psi_{n,k}$. Let $A$ be any $(k-1)$-subset of $X$ and let
$\alpha = \bigwedge_{x \in A} x$.
By Lemma \ref{lem:directed-graph-G-A}, reachability
of the vertices in the graphs $G^d(\Psi_{n,k}, A)$ and $G^d(\phi, A)$
represent propagating assignments of the variables in $X \setminus A$
to $0$ in the formulas $\Psi_{n,k} \wedge \alpha$ and
$\phi \wedge \alpha$, respectively. By assumption, these formulas
are ucp-equivalent, so the two graphs have the same
transitive closure. Since the graph $G^d(\Psi_{n,k}, A)$ is the complete
directed graph on $X \setminus A$, the graph $G^d(\phi, A)$
is strongly connected.

Assume that $G^d(\phi, A)$ is strongly connected.
Let $C \in \Psi_{n,k} \setminus \phi$ and let us prove that $C$
is absorbed by $\phi$. Let $l \in C$ and consider a partial
assignment $\gamma$ of the variables falsifying all the literals
in $C \setminus \{l\}$. We have to show that UCP on
$\phi \wedge \gamma$ derives $l$.

If $l$ is positive, let $g \in C \setminus \{l\}$
be arbitrary. If $l$ is negative, let $g$ be the positive literal
contained in $C$. For a suitable $(k-1)$-set of variables $A$,
we have $C = \n(A) \cup \{g, l\}$ and $g$ and $l$ have the opposite sign.
Since $\gamma$ falsifies all the literals in $\n(A)$ and $g$, we have
$\gamma = \alpha \wedge \neg g$ where $\alpha = \bigwedge_{x \in A} x$.

Assume that the variable in $g$ is $x_i$ and the variable in $l$ is $x_j$.
Hence, we have either $g=x_i$ and $l=\neg x_j$ or $g=\neg x_i$ and $l=x_j$.
By assumption, $G^d(\phi, A)$ contains paths between $x_i$ and $x_j$
in both directions. If $g=x_i$, then $\gamma = \alpha \wedge \neg x_i$
and Lemma~\ref{lem:directed-graph-G-A}
implies that UCP on $\phi \wedge \gamma$ derives $\neg x_j = l$.
If $g=\neg x_i$, then $\gamma = \alpha \wedge x_i$
and Lemma~\ref{lem:directed-graph-G-A}
implies that UCP on $\phi \wedge \gamma$ derives $x_j=l$.

We obtained that every clause $C \in \Psi_{n,k} \setminus \phi$
is absorbed by $\phi$. It follows that $\Psi_{n,k}$ is absorbed by $\phi$.
Since $\phi$ is a subset of $\Psi_{n,k}$, it is ucp-equivalent
to $\Psi_{n,k}$.
\end{proof}

\begin{proposition} \label{prop:lower-bound}
The size (number of clauses) of an arbitrary CNF formula ucp-equivalent
to $\Psi_{n,k}$ is at least $n\, C(n-1, k, k-1)$.
\end{proposition}

\begin{proof}
Let $\phi^*$ be a smallest CNF formula ucp-equivalent to $\Psi_{n,k}$.
By Lemma \ref{lem:prime}, without
loss of generality, we can assume that $\phi^*$ consists only of its
prime implicates and, hence, is a subset of $\Psi_{n,k}$.
For every variable $x \in X$, let
$$
\mathcal{S}_x = \{B \mid \n(B) \vee x \in \phi^*\} \;.
$$
Clearly, $\mathcal{S}_x$ is a system of $k$-sets
$B \subseteq X \setminus \{x\}$.

Let $x \in X$ and let $A$ be an arbitrary $(k-1)$-subset
of $X \setminus \{x\}$.
By Proposition~\ref{prop:propagation}, vertex $x$ has non-zero in-degree
in the graph $G^d(\phi^*, A)$, so there is a vertex $y \in X \setminus A$
different from $x$ such that $(y, x)$ is an edge of the graph.
By construction of $G^d(\phi^*, A)$, $\n(A) \cup \{\neg y, x\}$ is
a clause of $\phi^*$. The set $B = A \cup \{y\}$ is a member of
$\mathcal{S}_x$ and $A \subseteq B$. It follows that $\mathcal{S}_x$
is a system
of $k$-subsets of an $(n-1)$-set covering all $(k-1)$-subsets.
Hence, $|\mathcal{S}_x| \ge C(n-1,k,k-1)$ and, since
$|\phi^*| = \sum_{x \in X} |\mathcal{S}_x|$, the statement follows.
\end{proof}

Let us recall that both the formulas $\phi^\ell$ and $\phi^*$
needed for the main result will be constructed as hypergraph
formulas.
Proposition~\ref{prop:propagation} characterizes general subsets
of $\Psi_{n,k}$ which are ucp-equivalent to the whole formula
using the directed graphs $G^d(\phi, A)$. It appears that for
hypergraph formulas $\theta(H) \subseteq \Psi_{n,k}$ the following
simpler condition using an undirected graph is sufficient for the
same purpose.

\begin{definition} \label{def:conn-restriction}
We say that a $(k+1)$-uniform hypergraph $H$ on the set of vertices $X$
has connected restrictions, if for every $(k-1)$-set $A \subseteq X$,
the graph
\begin{equation} \label{eq:hypergraph-restriction}
G(H, A) = \{e \setminus A \mid A \subseteq e \in H\}
\end{equation}
is a connected graph on the set $X \setminus A$.
\end{definition}

\begin{proposition} \label{prop:comb-reformulation}
If $H$ is a $(k+1)$-uniform hypergraph on $X$, then $\theta(H)$ is
ucp-equivalent to $\Psi_{n,k}$ if and only if $H$ has connected restrictions.
\end{proposition}

\begin{proof}
Since $\theta(H)$ is a subset of $\Psi_{n,k}$, we can use the
characterization of formulas ucp-equivalent to $\Psi_{n,k}$ from
Proposition~\ref{prop:propagation}.

Let $A$ be an arbitrary $(k-1)$-subset of $X$. By construction
of $\theta(H)$, if $G^d(\theta(H), A)$
contains the directed edge $(x_i, x_j)$, then it also contains the
directed edge $(x_j, x_i)$. Moreover, these two edges
are contained in $G^d(\theta(H), A)$ if and only if $G(H, A)$
contains the undirected edge $\{x_i, x_j\}$. Hence, the graph
$G^d(\theta(H), A)$ is strongly connected if and only if $G(H, A)$ is
connected. Since this is satisfied for arbitrary $(k-1)$-subset
$A \subseteq X$, the statement follows.
\end{proof}

Note that $\Psi_{n,k}$ is a hypergraph formula, since it is $\theta(H_0)$,
where $H_0$ is the complete $(k+1)$-uniform hypergraph on $X$.
The ucp-irredundant
formula $\phi^\ell$ can be obtained from an easy to define subset of $H_0$.

\begin{definition} \label{def:phi-ell}
Let $H_1$ consist of all $(k+1)$-subsets of $X$ containing the
variable $x_1$ and let $\phi^\ell=\theta(H_1)$.
\end{definition}

The proportion of the $(k+1)$-subsets of $X$ containing the
variable $x_1$ in the set of all these subsets is $(k+1)/n$.
For the main result, we assume $k = n/2 + O(1)$, so the size
of $\phi^\ell$ is roughly one half of the size of $\Psi_{n,k}$.

\begin{proposition} \label{prop:phi-ell}
Formula $\phi^\ell$ is ucp-irredundant and ucp-equivalent to $\Psi_{n,k}$.
\end{proposition}

\begin{proof}
The hypergraph $H_1$ has connected restrictions, since for every $(k-1)$-set
$A$ not containing $x_1$, the graph $G(H_1,A)$ is a star on $X\setminus A$
whose central vertex is $x_1$. If $x_1 \in A$, then $G(H_1,A)$ is
a complete graph on $X\setminus A$. It follows
that $\phi^\ell$ is ucp-equivalent to $\Psi_{n,k}$ by
Proposition~\ref{prop:comb-reformulation}.

Let us prove that $\phi^\ell$ is ucp-irredundant. For this purpose,
let $C \in \phi^\ell$ be arbitrary and let us prove that
$\phi^\ell \setminus \{C\}$ is not ucp-equivalent to $\Psi_{n,k}$.
Let $V$ be the set of variables used in $C$ except of $x_1$,
so we have $|V|=k$. Let $\gamma$ be the partial assignment of the
variables in $V$ such that all the literals on variables from $V$ in $C$
are falsified. UCP on the clause $C$ together with $\gamma$ derives
a value of $x_1$ which makes the literal on $x_1$ in $C$ satisfied.
Since $C$ is an implicate of $\Psi_{n,k}$, every formula
ucp-equivalent to $\Psi_{n,k}$ should derive the same value of $x_1$
as $C$. The proof will be finished by proving that
$\phi^\ell \setminus \{C\}$ does not satisfy this.

Depending on whether $C$ contains the literal $x_1$ or $\neg x_1$,
the partial assignment $\gamma$ assigns all variables in $V$ the
value $1$ or all variables in $V$ are assigned the value $1$ except
of one which is assigned $0$. In both cases, the following is true.
Every clause $D$ in $\phi^\ell \setminus \{C\}$ on the same set of
variables as $C$ is satisfied by $\gamma$. The reason is that two
of the variables
in $C$ have different sign in $C$ and $D$. At least one of them is
not $x_1$. The literal with this variable in $C$ is falsified, so
its negation in $D$ is satisfied.
Every clause in $\phi^\ell \setminus \{C\}$ on a set of variables
different from the set of variables of $C$ contains at
least two variables not assigned by $\gamma$, namely $x_1$ and some
variable not in $V \cup \{x_1\}$. Such a clause cannot
be used for UCP. Altogether, UCP on
$\phi^\ell \setminus \{C\}$ with the assignment $\gamma$
does not derive a value of any new variable, in particular,
a value of $x_1$.
\end{proof}

Construction of a hypergraph $H$ such that $\phi^* = \theta(H)$
is a sufficiently small formula ucp-equivalent to $\Psi_{n,k}$ is
more complicated
and we use results proven in the next two sections for this purpose.

\section{Union of random permutations of a graph}
\label{sec:random-permutations}

In the next section we prove the existence of a sufficiently small
$(k+1)$-uniform hypergraph on the set $X$ with connected restrictions,
so that Proposition~\ref{prop:comb-reformulation} then allows to
prove existence of the required formula $\phi^*$. Using known results on
covering systems, it is possible to construct a hypergraph $H$ satisfying
a weaker condition, namely, that the graph $G(H, A)$ has no isolated
vertices for every $A \subseteq X$ of size $k-1$. A possible way
how to extend a graph with no isolated vertices to a connected
graph is to take a union of a constant number of its copies obtained by
a random permutation of the vertices.
The purpose of this section is to prove a small upper bound on the
probability that this union fails to be connected.
If this probability is small enough, then the process can be used
in parallel for almost all sets $A$ which we consider and allows
to prove the existence of the required hypergraph.

The following definition
will be used in the next section for hypergraphs and this section
uses its special case for graphs which are $2$-uniform hypergraphs.

\begin{definition} \label{def:permutation}
Let $H$ be a hypergraph on a set of vertices $V$. Consider
a permutation $\pi:V \to V$. We denote
by $\pi(H)$ the hypergraph on the set of vertices $V$
such that for every set $e \subseteq V$, $\pi(e)$ is a hyperedge
of $\pi(H)$ if and only if $e$ is a hyperedge of $H$.
\end{definition}

Intuitively, graph or hypergraph $\pi(H)$ is obtained by moving each
vertex $u$ of $H$ to $\pi(u)$ and moving the edges correspondingly.
By a random permutation of $H$, we mean $\pi(H)$, where $\pi$ is
chosen at random from the uniform distribution on permutations of $V$.

An undirected graph $G=(V,E)$ is connected if and only if for every
partition of $V$ into non-empty sets $Z$ and $V \setminus Z$, there
is an edge connecting a vertex in $Z$ and a vertex in $V \setminus Z$.
We need an upper bound on the probability of the complement
of this event for a fixed set $Z \subseteq V$ and a graph obtained
as a random permutation of a given graph. Equivalently, we can
consider a fixed graph and a random set $Z$ of a given size. Clearly,
there is no edge between $Z$ and $V \setminus Z$ if and only if $Z$
is a union of the connected components of the graph.

Assume a graph on the set of vertices $V$ of size $m$ with
no isolated vertices and with $r$ connected components of sizes
$a_i \ge 2$, $i=1,\ldots,r$.
The following lemma provides an upper bound on the number of
subsets $Z$ of $V$ of size $d$ which are unions of connected
components of the graph. In the simple case when $m$ and $d$ are even
and $a_i=2$ for $i=1,\ldots,m/2$, the number of these sets is exactly
${m/2 \choose d/2}$. In the general case, a similar estimate can
be obtained. The considered subsets of $V$ are encoded by
the sets $I$ satisfying (\ref{eq:sum-equal-d}). Since the sets $I$
are pairwise incomparable we can use LYM inequality \cite{Lub66} implying an
upper bound on the number of sets of different sizes in a collection
of pairwise incomparable sets.

\begin{lemma}[\cite{Sil23}] \label{lem:integer-sums}
Assume integers $a_i \ge 2$ for $i=1,\ldots, r$ and let $m=\sum_{i=1}^r a_i$.
Then for every $2 \le d \le \lfloor m/2 \rfloor$, the number of sets
$I \subseteq \{1, \ldots, r\}$, such that
\begin{equation} \label{eq:sum-equal-d}
\sum_{i \in I} a_i = d
\end{equation}
is at most
$$
{\lfloor m/2 \rfloor \choose \lfloor d/2 \rfloor} \;.
$$
\end{lemma}

\begin{proof}
The sets $I$ satisfying (\ref{eq:sum-equal-d}) are pairwise incomparable
and have size at most $\lfloor d/2 \rfloor$.
Let $c_j$ for $j=1, \ldots, \lfloor d/2 \rfloor$ be the number of these
sets of size $j$. By LYM inequality \cite{Lub66}, we have
$$
\sum_{j=1}^{\lfloor d/2 \rfloor} \frac{c_j}{{r \choose j}} \le 1 \;.
$$
Since $r \le \lfloor m/2 \rfloor$, we have also
$$
\sum_{j=1}^{\lfloor d/2 \rfloor} \frac{c_j}{{\lfloor m/2 \rfloor \choose j}} \le 1
$$
and, since $\lfloor d/2 \rfloor \le \lfloor m/2 \rfloor / 2$, this implies
that the number of sets $I$ satisfying (\ref{eq:sum-equal-d}) is at most
$$
\sum_{j=1}^{\lfloor d/2 \rfloor} c_j \le
{\lfloor m/2 \rfloor \choose \lfloor d/2 \rfloor}
$$
as required.
\end{proof}

\begin{lemma} \label{lem:not-shattered}
Let $G=(V,E)$ be an undirected graph with no isolated vertices
where $|V|=m$, let $2 \le d \le m/2$, and let $Z$ be a subset of
$V$ of size $d$.
The probability that a random permutation $\pi(G)$ of $G$
has no edge connecting $Z$ and $V \setminus Z$ is at most
$$
{\lfloor m/2 \rfloor \choose \lfloor d/2 \rfloor}
{m \choose d}^{-1} \;.
$$
\end{lemma}

\begin{proof}
Consider a random permutation $\pi$ of the vertices of $G$. Graph $\pi(G)$
has no edge between $Z$ and $V\setminus Z$ if and only if the preimage
$\pi^{-1}(Z)$ is a union of connected components of $G$.
The preimages $\pi^{-1}(Z)$ are uniformly distributed over subsets of $V$
of size $d$. The number of subsets of $V$ of size $d$ that are a union
of connected components in $G$ is equal to the number of solutions of the
equation (\ref{eq:sum-equal-d}) where $r$ is the number of connected
components of $G$ and $a_i$ are their sizes. Since $G$ has no isolated
vertices, we have $a_i \ge 2$ for $i=1,\ldots,r$ and,
hence, the statement follows from Lemma~\ref{lem:integer-sums}.
\end{proof}

\begin{proposition} \label{prop:union-of-permutations}
Let $G_1, \ldots, G_s$, $s \ge 3$, be graphs on the same set of $m$
vertices each of which has no isolated vertices.
If $s$ is fixed and $m$ tends to infinity, then the probability that
the union of independent random permutations of $G_1, \ldots, G_s$
is not a connected graph is at most $O(m^{2-s})$.
\end{proposition}

\begin{proof}
Let $V$ be the common set of vertices of the graphs.
The probability of the event that the union of independent permutations
of $G_1, \ldots, G_s$ is not a connected graph is at most the sum of the
probabilities over subsets $Z \subseteq V$ of size between $2$ and $m/2$
of the event that the union contains no edge between $Z$ and $V \setminus Z$.
Using $d=|Z|$ as a summation index and using Lemma~\ref{lem:not-shattered},
this is at most
$$
\sum_{d=2}^{\lfloor m/2 \rfloor} p_d
$$
where
\begin{equation} \label{eq:p_d:1}
p_d =
{\lfloor m/2 \rfloor \choose \lfloor d/2 \rfloor}^s {m \choose d}^{1-s} \;.
\end{equation}
For a few smallest values of $d$, the following simple estimate
is sufficient.
If $d$ and $s$ are fixed and $m\to \infty$, we have
\begin{equation} \label{eq:p_d:2}
p_d = O\left(m^{sd/2 + (1-s)d} \right) = O\left(m^{(2-s)d/2}\right) \;.
\end{equation}
For larger values of $d$, we estimate $p_d$ using a more precise bound
on binomial coefficients. Denote
$$
b(t, d) =
\left(\frac{t}{d}\right)^{d}
\left(\frac{t}{t-d}\right)^{t-d}
= \mathrm{e}^{t \entropy(d/t)}
$$
where $\entropy(x) = -x \ln x - (1-x) \ln(1-x)$ is the entropy function.
Since $\entropy(x)$ is increasing for $0 \le x \le 1/2$,
$b(t,d)$ is increasing in $d$ if $t$ is fixed and $1 \le d \le t/2$.
Since $\entropy(x)$ is strictly concave and $\entropy(0)=0$,
$\entropy(x)/x$ is decreasing for $0 < x \le 1$. By substitution
$x=d/t$, we obtain that $b(t,d)$ is increasing in $t$ if $d$ is fixed.
By a well-known upper bound on binomial coefficients and
using $d \le m/2$, we get
$$
{\lfloor m/2 \rfloor \choose \lfloor d/2 \rfloor} \le
b(\lfloor m/2 \rfloor, \lfloor d/2 \rfloor) \le
b(m/2, d/2) = b(m, d)^{1/2} \;.
$$
By a well-known lower bound on binomial coefficients, we have
$$
\frac{1}{m+1} b(m, d) \le {m \choose d} \;.
$$
Using these bounds and (\ref{eq:p_d:1}), we obtain
\begin{equation} \label{eq:p_d:3}
p_d \le (m+1)^{s-1} b(m, d)^{1-s/2} \;.
\end{equation}
Let $d_0 = \lceil 4(s-1)/(s-2) \rceil$. Using (\ref{eq:p_d:2})
for $d \ge 2$, we obtain
\begin{equation} \label{eq:p_d:4}
\sum_{d=2}^{d_0-1} p_d = O(m^{2-s}) \;.
\end{equation}
If $d_0 \le d \le m/2$, then
$b(m,d) \ge b(m,d_0) \ge (m/d_0)^{d_0}$ and by (\ref{eq:p_d:3})
$$
p_d \le (m+1)^{s-1} (m/d_0)^{d_0(1-s/2)} = O(m^{1-s}) \;.
$$
Together with (\ref{eq:p_d:4}), we obtain
\begin{equation}
\sum_{d=2}^{\lfloor m/2 \rfloor} p_d = O(m^{2-s})
\end{equation}
and the proof is finished.
\end{proof}

\section{Size of a hypergraph with connected restrictions}
\label{sec:size-of-hypergraph}

The formula $\phi^*$ needed for the main result will be constructed
as a hypergraph formula. By Proposition \ref{prop:comb-reformulation},
in order to obtain the required formula, we need a sufficiently small
hypergraph with connected restrictions. In this section, we prove the
existence of such a hypergraph whose size is only by a constant factor
larger than the size of a suitable covering system.

We start from a covering system $H$ and show that as a hypergraph
it satisfies a condition related to but weaker than having connected
restrictions. As suggested in the previous section, we consider several
independent random permutations of its vertices. With high probability,
the union $H^*$ of the permuted hypergraphs has almost all restrictions
connected or, more formally, $G(H^*, A)$ is connected for almost all
sets $A$ of size $k-1$. The final
hypergraph $H^{**}$ is then obtained by extending $H^*$ with a limited
number of hyperedges so that it has connected restrictions and the
required upper bound on its size is achieved.

\begin{proposition} \label{prop:small-H-conn-restr}
For every large enough $n$ and an integer $k = n/2 + O(1)$, there is
a $(k+1)$-uniform hypergraph on $n$ vertices of size at most
$6\, C(n,k+1,k)$ with connected restrictions.
\end{proposition}

\begin{proof}
Let $|X|=n$ and let $H$ be a $(k+1)$-uniform hypergraph of size
$C(n, k+1, k)$ covering all $k$-subsets of $X$. Let $s=5$,
let $\pi_1, \ldots, \pi_s$ be independent random permutations
of $X$ chosen from the uniform distribution, and let
$H^* = \pi_1(H) \cup \ldots \cup \pi_s(H)$.

If $A$ is a $(k-1)$-subset of $X$ and $x \in X \setminus A$,
then by assumption, there is a hyperedge $B \in H$ containing
$A \cup \{x\}$, so $B = A \cup \{x, y\}$ for some $y \in X \setminus A$.
In particular, vertex $x$ has non-zero degree in $G(H, A)$.
This argument applies also to each of the graphs $G(\pi_i(H),A)$,
so each of them has
no isolated vertices. The graph $G(H^*, A)$ is a union of $s$
graphs obtained in this way, however, we cannot directly use
Proposition~\ref{prop:union-of-permutations}.
The graph $G(\pi_i(H),A)$ is isomorphic to $G(H,\pi_i^{-1}(A))$
and, hence, it is not, in general, a random permutation
of a fixed graph on $X \setminus A$.
In order to overcome this, we partition the probability space of
all $s$-tuples of permutations of $X$ by a set of mutually exclusive
conditions such that the assumption of
Proposition~\ref{prop:union-of-permutations} is satisfied by
the conditional distribution of the $s$-tuple of graphs
$G(\pi_i(H),A)$, $i=1,\ldots,s$ under any of these conditions.
Namely, for an arbitrary $s$-tuple $(A_1', \ldots, A_s')$ of $(k-1)$-subsets
of $X$ let $c(A_1', \ldots, A_s')$ be the condition $\pi_i(A_i') = A$
for all $i=1,\ldots,s$. Clearly, these conditions represent a partition
of the probability space and, since $\pi_i$ are independent and
uniformly distributed, all of these conditions have the same probability.

Fix an arbitrary $s$-tuple $(A_1', \ldots, A_s')$ and consider the
conditional distribution under $c(A_1', \ldots, A_s')$.
For a fixed $i$ and all permutations $\pi_i$ satisfying $\pi_i(A_i') = A$,
the graphs $G(\pi_i(H),A)$ are isomorphic, since all of them are isomorphic
to $G(H,A_i')$. Choose any of the
graphs $G(\pi_i(H),A)$ as $G_i$. The distribution of $G(\pi_i(H),A)$ is
the same as the distribution of random permutations of $G_i$. Moreover,
these random permutations
are independent for $i=1,\ldots,s$. Hence, we can use
Proposition~\ref{prop:union-of-permutations} with $V=X \setminus A$
to obtain that $G(H^*, A)$ is not connected with probability at most
$O((n-k+1)^{2-s})$ under any of the considered conditions and, hence,
also unconditionally.

A $(k-1)$-set $A$ is called ``bad'', if the graph $G(H^*,A)$ is not connected
and let $t$ be the random variable equal to the number of such sets.
By the argument above, we have
\begin{equation} \label{eq:expected-bad-sets}
\mathbf{E}\; t = O\left((n-k+1)^{2-s}\right) {n \choose k-1} \;.
\end{equation}
For each of the $t$ ``bad'' sets $A$, we add several $(k+1)$-sets to $H^*$
in order to add edges to $G(H^*,A)$ so that this graph becomes connected.
The number of $(k+1)$-sets needed for this is at most the number
of the connected components of $G(H^*,A)$ minus $1$
which is at most $(n-k-1)/2$. Hence, the resulting $(k+1)$-uniform hypergraph
$H^{**}$ has size at most $s |H| + t (n-k-1)/2$. By (\ref{eq:expected-bad-sets}),
there is a choice of permutations $\pi_i$ for $i=1,\ldots,s$, such that
$$
t = O\left((n-k+1)^{2-s}\right) {n \choose k-1}
= O\left((n-k+1)^{2-s}\right) {n \choose k}
$$
Since $n-k = \Theta(k)$, we can conclude
$$
|H^{**}| \le s |H| + O\left(k^{3-s}\right) {n \choose k} \;.
$$
Using (\ref{eq:trivi-lower-bound}) and $s=5$, we get
$$
|H^{**}| \le (s+1) |H| \le 6 \, C(n, k+1, k)
$$
if $n$ and, hence, also $k$ are sufficiently large. This finishes
the proof, since $H^{**}$ has connected restrictions.
\end{proof}

\section{Proof of the results}
\label{sec:results-proof}

In this section, we summarize the results of the previous sections.
First, we finish the proof of Proposition~\ref{prop:size-of-theta}
on a lower and an upper bound on the minimum size of a formula
ucp-equivalent to $\Psi_{n,k}$.

\begin{proof}[Proof of Proposition~\ref{prop:size-of-theta}]
The lower bound is Proposition~\ref{prop:lower-bound}.
For the upper bound, assume a large enough $n$ and $k = n/2 +O(1)$.
Our goal is to prove that there is a formula ucp-equivalent to
$\Psi_{n,k}$ and of size at most $6(k+1) C(n, k+1, k)$. For this purpose,
we use a hypergraph formula. Consider a hypergraph $H$ of size at most
$6\, C(n, k+1, k)$ with connected restrictions
guaranteed by Proposition~\ref{prop:small-H-conn-restr}.
By Proposition~\ref{prop:comb-reformulation}, $\theta(H)$ is
ucp-equivalent to $\Psi_{n,k}$ and its size is $(k+1)|H|$
as required. Using
(\ref{eq:define-mu-C}) and the bound (\ref{eq:mu-upper-1})
or (\ref{eq:mu-upper-2}), we obtain the required asymptotic
estimate.
\end{proof}

The constructions presented up to now imply
Theorem~\ref{thm:separation} as follows.

\begin{proof}[Proof of Theorem~\ref{thm:separation}]
Let $\phi^\ell$ be the formula from Definition~\ref{def:phi-ell}.
This formula is ucp-irredundant and ucp-equivalent to $\Psi_{n,k}$ by
Proposition~\ref{prop:phi-ell}. Assuming $k = n/2 + O(1)$,
its size is
$$
|\phi^\ell| = (k+1) {n-1 \choose k}
= (k+1)\left(1 - \frac{k}{n}\right) {n \choose k}
= \Theta(n) {n \choose k} \;.
$$
Let $\phi^* \subseteq \Psi_{n,k}$ be a formula ucp-equivalent to $\Psi_{n,k}$
of size satisfying the upper bound from Proposition~\ref{prop:size-of-theta}.
The statement follows by comparing the size of the formulas $|\phi^\ell|$
and $|\phi^*|$.
\end{proof}

\section{Conclusion and further research}
\label{sec:conclusion}

The paper introduces ucp-equivalence on CNF formulas and investigates
the ratio of the size of a ucp-irredundant formula and the size of
a smallest ucp-equivalent formula. The paper shows that a general upper
bound on this ratio is at most $n^2$ (Proposition~\ref{prop:upper-bound})
and at least $\Omega(n/\ln n)$ (Theorem~\ref{thm:separation}). The gap
between these bounds is quite large. An interesting open problem
is to reduce this gap. It is plausible to assume that the upper
bound can be improved. Another interesting question is whether
the lower bound $\Omega(n/\ln n)$ or an even better bound can be proven
using a different type of functions and without the use
of probabilistic method.

\bigskip
\noindent\textbf{Acknowledgement.} The author would like to thank
Diana Piguet for a fruitful discussion on using hypergraphs, in particular,
for suggesting Definition~\ref{def:use-undirected-hg} which later
appeared to be used also for other purposes, and
Matas {\v S}ileikis for suggesting Lemma~\ref{lem:integer-sums}
which improves a slightly weaker bound on the same quantity by the author.
The author acknowledges institutional support RVO:67985807.

\bibliography{ucp_irredundant}{}
\bibliographystyle{unsrt}

\end{document}